\documentclass[graybox]{svmult}

\usepackage{amsfonts,amstext,amssymb,physics} 
\usepackage{amsmath}

\usepackage{amssymb}
\usepackage{mathptmx}       
\usepackage{helvet}         
\usepackage{courier}        
\usepackage{type1cm}        
\usepackage{graphicx}        
\usepackage{multicol}        
\usepackage[bottom]{footmisc}

\usepackage{enumerate}

\usepackage{enumitem}  
\usepackage{calc}
\setlist{labelindent=1pt,itemsep=0.1cm}
\setlist[itemize]{leftmargin=0.7cm}
\setlist[enumerate]{itemindent=0.2em,leftmargin=0.7cm}

\usepackage{float} 

\usepackage[bookmarks=false,hidelinks,unicode]{hyperref}
\hypersetup{hypertexnames=false}  

\smartqed

\allowdisplaybreaks

\makeindex             

\begin{document}
\title*{Common end points of multivalued mappings in ordered metric spaces}
\titlerunning{Common end points of multivalued mappings in ordered metric spaces} 
\author{Talat Nazir \and Sergei Silvestrov}
\authorrunning{T. Nazir, S. Silvestrov} 

\institute{Talat Nazir,
\at
Department of Mathematical Sciences, University of South Africa, Florida 0003, South Africa. \\ \email{talatn@unisa.ac.za}
\and
Sergei Silvestrov
\at Division of Mathematics and Physics, School of Education, Culture and Communication, M{\"a}lardalen University, Box 883, 72123 V{\"a}ster{\aa}s, Sweden. \\  \email{sergei.silvestrov@mdu.se}
\and }
%
%


\maketitle
\label{chap:NazirSilvestrov:CEPOMS}

\abstract*{Using the setting of ordered metric spaces, we obtain common end point of two multivalued mappings satisfying a generalized $(\psi,\varphi)$-weak contractive condition. Under comparative condition on the set of end points of multivalued mappings, our results assure the uniqueness of the end point. These results generalize and improve several recent results on single-valued as well as multivalued mappings.
}

\abstract{Using the setting of ordered metric spaces, we obtain common end point of two multivalued mappings satisfying a generalized $(\psi,\varphi )$-weak contractive condition. Under comparative condition on the set of end points of multivalued mappings, our results assure the uniqueness of the end point. These results generalize and improve several recent results on single-valued as well as multivalued mappings.
}

\section{Introduction and Preliminaries}
\label{secNaSe:Introductionconemetricspace}

In 2005, Nieto and Rodrigues-L\'{o}pez \cite{Nieto,Nieto2} proved a
modified variant of result of Ran and Reuring \cite{Ran}. Regan and Petru
\c{s}el \cite{Regan} proved fixed point theorems for generalized
contractions in ordered metric spaces. Also, Nieto et al. \cite{Nieto1}
improved some results given by Petru\c{s}el and Rus \cite{Petrusel}, in the
setting of abstract $L$-spaces in the sense of Fr\'{e}chet \cite{Frechet}. Agarwal et al. \cite{Agarwal} obtained fixed point results of generalized contractions in partially ordered metric spaces.
On the other hand, Alber and Guerre-Delabriere \cite{Alber} defined weakly
contractive maps on a Hilbert space and established a fixed point theorem
for such map. Afterwards, Rhoades \cite{Rhoades} using the notion of weakly
contractive maps, obtained a fixed point theorem in a complete metric space.
Dutta et al. \cite{Dutta} generalized the weak contractive condition and
proved a fixed point theorem for a self-map, which in turn generalizes
\cite[Theorem 1]{Rhoades} and the corresponding result in \cite{Alber}. The study
of common fixed points of mappings satisfying certain contractive conditions
has been at the center of vigorous research activity. Abbas and Khan
\cite{Abbas} extended the result of Dutta \cite{Dutta} to two mappings. Zhang and
Song \cite{Zhang} introduced the concept of a generalized $\varphi$-weak
contraction condition and obtained a common fixed point for two maps. Doric
\cite{Doric} proved a common fixed point theorem for generalized $(\psi
,\varphi )$-weak contractions. There are many results in the existing
literature which deal with fixed point of multivalued mappings (see
\cite{Jungck}). In some cases multivalued mapping $T$ defined on a nonempty set
$X $ assumes a compact value $Tx$ for each $x$ in $X.$ There are the
situations when for each $x$ in $X,$ $Tx$ is assumed to be closed and
bounded subset of $X.$ To prove the existence of fixed point of such
mappings, it is essential for mappings to satisfy certain contractive
conditions which involve Hausdorff metric. The metric in our case is defined
in terms of diameters of sets. Recently, Abbas et al. \cite{AbbasNazir}
obtained the common fixed points results of multivalued Perov type
contractions on cone metric spaces. Khan et al. \cite{KhanAbbas} established
some\ fixed points of multivalued contractions in the setup of partial
metric spaces.

The aim of this article is to obtain the common end points of two
multivalued mappings without appeal to continuity of any map involved
therein in the framework of ordered metric spaces. It is also noted that our
result do not require any commutativity condition to prove an existence of
common end point of two mappings. These results extend, unify and improve
existing comparable results in the existing literature.

Throughout this work, we denote $\mathbb{R}_{\geq 0}$ a set of non negative real numbers and $\mathbb{R}_{> 0}$ a set of positive real numbers. Following definitions and results will be needed in the sequel.

\begin{definition}
	\label{DefinNazSe:1.}
	Let $X$ be a nonempty set. The relation $"\preceq "\,$on $X$ is called a
	partial order if it satisfies the following conditions:
	\begin{enumerate}[label=\textup{\arabic*)}, ref=\arabic*]
		\item $x\preceq x$ for all $x\in X.$		
		\item $x\preceq y$ and $y\preceq x$ implies $x=y$ for all $x,y\in X.$		
		\item $x\preceq y$ and $y\preceq z$ implies $x\preceq z$ for $
		x,y,z\in X.$
	\end{enumerate}
A set with partial order $\preceq $ is called a partially ordered set. We
	will denote
$$X_{\preceq }:=\{(x,y)\in X^2:x\preceq y \ \text{or}\ y\preceq x \}, \quad s(X):=\{ \{x_{n}\}:x_{n}\in X, n\geq 1\}.$$
\end{definition}	
	Consider a subset $c(X)$ of $s(X)$ and a mapping $Lim:c(x)\rightarrow X.$
Then the triplet $(X,c(X),Lim)$ is called an $L$-space if the following
conditions are satisfied \cite{Frechet}:
\begin{enumerate}[label=\textup{(\roman*)}, ref=(\roman*)]
	\item If $x_{n}=x,$ for all $n\geq 1,$ then $\{x_{n}\} \in c(X)$ and
    $Lim\{x_{n}\}=x.$
	\item If $\{x_{n}\} \in c(X)$ and $Lim\{x_{n}\}=x,$ then $
	\{x_{n_{k}}\} \in c(X)$, for any subsequence $\{x_{n_{k}}\}$ of $\{x_{n}\}$
	and $Lim\{x_{n_{i}}\}=x.$
\end{enumerate}

 By definition, $c(X)$ contains convergent sequences in $X,$ $x:=Lim\{x_{n}\}$
is the limit of $\{x_{n}\}$ and we write $x_{n}\rightarrow x $ as $
n\rightarrow \infty .$ We denote $L$-space by $(X,\rightarrow ).$

\begin{definition}
	\label{DefinNazSe:2.}
	Let $X$ be a nonempty set. Then $X$ is called an ordered $L$-space if and
	only if the following conditions are satisfied:
	\begin{enumerate}[label=\textup{(\roman*)}, ref=(\roman*)]
		\item $(x,\rightarrow )$ is an $L$-space.	
		\item $(X,\preceq )$ is a partially ordered set.	
		\item $(x_{n}\rightarrow x\ \text{and}\ y_{n}\rightarrow y\ \text{with}\
		x_{n}\preceq y_{n},\ \text{for each}\ n\geq 1) \ \Rightarrow\ (x\preceq y).$
	\end{enumerate}
	We denote the ordered $L$-space by $(X,\rightarrow ,\preceq ).$
\end{definition}

\begin{definition}
	\label{DefinNazSe:3.}
	A nonempty set $X$ is called an ordered $L_{\downarrow}$-space if and only if:
	
	\begin{enumerate}[label=\textup{(\roman*)}, ref=(\roman*)]
		\item $(x,\rightarrow )$ is an $L$-space.		
		\item $(X,\preceq )$ is a partially ordered set.		
		\item $(x_{n}\rightarrow x\ \text{and}\ x_{n+1}\preceq x_{n},\ \text{for each}\ n\geq 1)\ \Rightarrow\ (x\preceq x_{n}\ \text{for each}\ n\geq 1).$
	\end{enumerate}
	We denote the ordered $L_{\downarrow}$-space by $(X,\rightarrow ,\preceq_{\downarrow}).$
\end{definition}

\begin{definition}
	\label{DefinNazSe:4.}
	A nonempty set $X$ is called an ordered $L_{\uparrow}$-space if and only if
	\begin{enumerate}[label=\textup{(\roman*)}, ref=(\roman*)]
		\item $(x,\rightarrow )$ is an $L$-space.	
		\item $(X,\preceq )$ is a partially ordered set.		
		\item $(x_{n}\rightarrow x \ \text{and}\ x_{n}\preceq x_{n+1},
        \ \text{for each}\ n\geq 1)\ \Rightarrow\ (x_{n}\preceq x\ \text{for each}\ n\geq 1).$
	\end{enumerate}
	We denote the ordered $L_{\uparrow}$-space by $(X,\rightarrow ,\preceq_{\uparrow}).$ If $
	(X,d)$ is a metric space, then the convergence structure is given by the
	metric $"d"$ and the triplet $(X,d,\preceq_{\downarrow})$ (respectively $
	(X,d,\preceq_{\uparrow})$) will be called ordered $L_{\downarrow}$ metric space
	(respectively ordered $L_{\uparrow}$ metric space).
\end{definition}

\begin{example}
	\label{ExamNazSe:1.} The Euclidean plane $	\mathbb{R}
	^{2}$ with partial order defined as $(a,b)\preceq (c,d),$ $a,b,c,d\in
	\mathbb{R}
	$ if and only if $a\leq c$ and $b\leq d,$ where $"\leq "$ is usual order in $
	\mathbb{R}
	.$ Then $
	\mathbb{R}
	^{2}$ is an ordered $L$-space, ordered $L_{\downarrow}$-space and ordered
$L_{\uparrow}$-space.
\end{example}

\begin{example}
	\label{ExampNazSe:2.}
	Consider the set $X=\{1,2,\dots\}$ with the usual order $"\leq "$ and $
	c(X)=\{ \{x\}_{n\geq 1}:x\in X\} \cup A,$ where $A:=\{ \{n_{k}\}: \{n_{k}\} \ \text{is subsequence of}\ \{n\}_{n\geq 1} \}.$ Let $Lim:c(X)\rightarrow X$ be
	defined as
	\begin{equation*}
		Lim(z)=\left \{
		\begin{array}{rl}
			x, &\ \text{if} \ z=\{x\}_{n\geq 1} \\
			1, &\ \text{if} \ z\in A.
		\end{array}
		\right. \label{tequ1.1}
	\end{equation*}
	Then $(X,c(X),Lim)$ is an $L$-space with partial order $"\leq ".$ Clearly if
	$\{x_{n}\} \in A$, then $x_{n}\leq x_{n+1}.$ Therefore,
$(X,\rightarrow ,\leq_{\downarrow})$ with $\leq_{\downarrow}=\leq$
is an ordered $L_{\downarrow}$-space. Also if we take $\{x_{n}\}=\{2,2,2,\dots\}$ and $\{y_{n}\}=\{2,3,4,\dots\},$ then $
	x_{n}\rightarrow 2$, $y_{n}\rightarrow 1$ and $x_{n}\leq y_{n}$ but $2\nleq
	1.$ Note that $y_{n}\leq y_{n+1}$ but $y_{n}\nleq 1$ for any $n\geq 1.$
	Therefore $(X,c(X),Lim)$ is neither an ordered $L$-space nor an ordered $
	L_{\uparrow}$-space.
	
\end{example}

\begin{example}
	\label{ExampNazSe:3.}
	Consider the sets $X=\{-1,-2,\dots\}$ with the usual order $"\leq "$ and
	$c(X)=\{ \{x\}_{n\geq 1}:x\in X\} \cup A,$ where $A:=\{ \{n_{k}\}: \{n_{k}\}\ \text{is subsequence of}\ \{-n\}_{n\geq 1} \}.$ Let $Lim:c(X)\rightarrow X$
	be defined as
	\begin{equation*}
		Lim(z)=\left \{
		\begin{array}{rl}
			x, & \ \text{if} \ z=\{x\}_{n\geq 1} \\
			-1, & \ \text{if} \ z\in A.
		\end{array}
		\right.    
	\end{equation*}
	Then $(X,c(X),Lim)$ is an $L$-space with partial order $"\leq ".$ Clearly if
	$\{x_{n}\} \in A$ then $x_{n+1}\leq x_{n}.$ Therefore $(X,\rightarrow ,\leq_{\uparrow})$ with $\leq_{\uparrow}=\leq$ is an ordered $L_{\uparrow}$-space. Also, if we take $\{x_{n}\}=
	\{-2,-3,-4,\dots\}$ and $\{y_{n}\}=\{-2,-2,-2,\dots\},$ then $x_{n}\rightarrow
	-1 $, $y_{n}\rightarrow -2$ and $x_{n}\leq y_{n}$. Note that $-1\nleq -2$. Note also that $x_{n+1}\leq x_{n}$ but $-1\nleq x_{n}$ for any $n\geq 1.$ Therefore,
$(X,c(X),Lim)$ is neither an ordered $L$-space nor an ordered $L_{\downarrow}$-space.
\end{example}

The examples above show that $(X,\rightarrow ,\preceq ),$ $(X,\rightarrow
,\preceq_{\uparrow})$ and $(X,\rightarrow ,\preceq_{\downarrow})$ are three different
spaces.

Let $(X,d)$ be a metric space and $B(X)$ be the class of all nonempty
bounded subsets of $X$. We define the functions $\delta :B(X)\times
B(X)\rightarrow
\mathbb{R}_{\geq 0}$ and $D:B(X)\times B(X)\rightarrow
\mathbb{R}_{\geq 0}$ as follows
\begin{align*}
	\delta (A,B) &=\sup \{d(a,b):\;a\in A,\;b\in B\}, 
\\
	D(A,B) & =\inf \{d(a,b):\;a\in A,\;b\in B\}.  
\end{align*}
If $A$ contains a single point $a$, we write $\delta (A,B)=\delta (a,B)$. Also, if
$B$ contains a single point $b$, it yields $\delta (A,B)=d(a,b)$. Clearly, $
\delta (A,B)=\delta (B,A).$ For $\delta (\{a\},B)$ and $\delta (\{a\},\{b\})$
we write $\delta (a,B)$ and $d(a,b)$ respectively. We appeal to the fact
that $\delta (A,B)=0$ if and only if $A=B=\{x\}$ for $A,B\in B(X)$, $x\in X$ and $0\leq
\delta (A,B)\leq \delta (A,B)+\delta (B,C)$ for $A,B,C\in B(X).$

A point $x\in X$ is called a fixed point of $T$ if $x\in Tx.$ We shall
denote the set of all fixed points of $T$ by $F_{T}$ and the set of common
fixed points of two multivalued mappings $S$ and $T$ by $F_{S,T}.$ If there
exists a point $x\in X$ such that $Tx=\{x\},$ then $x$ is termed as an end
point of $T$.

\begin{definition}
	\label{DefinNazSe:5.}
	Let $X$ be a nonempty set. A mapping $T:X\rightarrow 2^{X}$ is said to be
	\begin{enumerate}[label=\textup{(\roman*)}, ref=(\roman*)]
		\item partially dominating if for every $y\in X$ there exist some
        $x\in Ty$ such that $y\preceq x;$	
		\item partially dominated if for every $y\in X$ there exist some $x\in
		Ty$ such that $x\preceq y.$
	\end{enumerate}
\end{definition}

\begin{example}
	\label{ExampleNazSe:6.}
Consider
$\mathbb{R}^{2}$ with partial order stated in Example \ref{ExamNazSe:1.}. Let $S,T,F:
	\mathbb{R}^{2}\rightarrow 2^{(\mathbb{R}^{2})}$ be the mappings defined by
	\begin{eqnarray*}
			S(x,y) & = & [(x+1,y),(x+2,y)], \\
			T(x,y) & = & [(x-2,y),(x-1,y)], \\  
			F(x,y) & = & [(x-1,y),(x+1,y)],
	\end{eqnarray*}
	where $[(x+1,y),(x+2,y)]$, $[(x-2,y),(x-1,y)]$, $[(x-1,y),(x+1,y)]$ are line segments between the points in
	$\mathbb{R}^{2}.$
It is straight forward to verify that $S$ is a partially dominating
	mapping but not partially dominated, $T$ is a partially dominated but not
	partially dominating, whereas $F$ is a partially dominating as well as
	partially dominated mapping.
\end{example}

\section{Common End Point Results}

 In this section, we obtain common end point results for
multivalued generalized $(\psi ,\varphi )$-weak contractive mappings defined
on a complete ordered\ metric spaces.

\begin{definition}
	\label{DefinNazSe:6.}
	Mappings $T,S:X\rightarrow B(X)$ are said to satisfy generalized
$(\psi,\varphi)$-weak contractive condition if the following inequality
\begin{align}
&		\psi (\delta (Sx,Ty))\leq \psi (M(x,y))-\varphi (M(x,y)), \label{tequ1.6} \\
& 	M(x,y)=\max \left \{ d(x,y),\delta (x,Sx),\delta (y,Ty),{\frac{1}{2}}(D(x,Ty)+D(y,Sx)) \right \}, \label{tequ1.7}
\end{align}
	holds for all $x,y\in X_{\preceq }\ $and for given functions
$\psi ,\varphi:\mathbb{R}_{\geq 0}\rightarrow \mathbb{R}_{\geq 0}$.
\end{definition}

Denote $\Upsilon $, the collection of all non-decreasing functions
$\mathbb{R}_{\geq 0}\rightarrow \mathbb{R}_{\geq 0}$ such that for every two sequences $
\{a_{n}\}$ and $\{b_{n}\}$ in $\mathbb{R}_{\geq 0}$ such that $\lim\limits_{n\rightarrow \infty }
a_{n}=\lim\limits_{n\rightarrow \infty }b_{n}$ it holds that $\lim\limits_{n\rightarrow \infty }\psi
(a_{n})=\lim\limits_{n\rightarrow \infty }\psi (b_{n})$.\newline

We start with the following result.
\begin{theorem}
	\label{TheoremNazSe:1}
	Let $(X,d,\preceq_{\downarrow})$ be a complete ordered $L_{\downarrow}$ metric space. Suppose
	that $T,S:X\rightarrow B(X)$ are two partially dominated mappings which
	satisfy generalized $(\psi ,\varphi )$-weak contractive condition, where
    $\psi \in \Upsilon $, and $\varphi :\mathbb{R}_{\geq 0}\rightarrow \mathbb{R}_{\geq 0}$
	satisfies, on the sequences $\{t_n\}\subset \mathbb{R}_{\geq 0}$, the property 	
    \begin{equation*}
		(\varphi (t_{n})\rightarrow 0\  \ as \ n\rightarrow \infty) \ \ \ \Rightarrow\  \  \ (t_{n}\rightarrow 0 \ as \ n\rightarrow \infty).
	\end{equation*}
	Then, there exists the point $u\in X$ such that $\{u\}=Tu=Su$. Moreover, if
	the end points of $S$ and $T$ are comparable, then $S$ and $T$ have a unique
	common end point.
\end{theorem}

\begin{proof} \smartqed We construct the convergent sequence $\{x_{n}\}$ in $X$ and
	prove that the limit point of that sequence is a unique common end point for
	$T$ and $S$. Let $x_{0}$ be an arbitrary point in $X$. By given assumption,
	there exist $x_{1}\in Sx_{0}$ such that $x_{1}\preceq x_{0}.$ Again there
	exist $x_{2}\in Tx_{1}$ such that $x_{2}\preceq x_{1}.$ Continuing this
	process, for each nonnegative integer $n$, we obtain
	\begin{equation*}
		x_{2n+1}\in Sx_{2n}=A_{2n}\ ,\quad x_{2n+2}\in Tx_{2n+1}=A_{2n+1}\  
	\end{equation*}
	\begin{equation*}
		\text{with}\ \ x_{n+1}\preceq x_{n}\ \text{for each}\ n. 
	\end{equation*}
	Also note that $x_{m}\preceq x_{n}$ for $m\geq n.$ Let
	\begin{equation}
		a_{n}=\delta (A_{n},A_{n+1}),\quad c_{n}=d(x_{n},x_{n+1}). \label{tequ1.11}
	\end{equation}
	
	Now we prove that $\{a_{n}\}$ and $\{c_{n}\}$ are convergent sequences.
	Suppose that $n$ is an odd integer. Substituting $x=x_{n+1}$ and $y=x_{n}$
	in \eqref{tequ1.6} and using properties of functions $\psi $ and $\varphi $, we obtain
	\begin{eqnarray*}
			\psi (\delta (A_{n+1},A_{n})) & = & \psi (\delta (Sx_{n+1},Tx_{n})) \\
			& \leq  & \psi (M(x_{n+1},x_{n}))-\varphi (M(x_{n+1},x_{n})) \leq  \psi (M(x_{n+1},x_{n})),
	\end{eqnarray*}
	which further implies that
	\begin{equation*}
		\delta (A_{n+1},A_{n})\leq M(x_{n+1},x_{n}). 
	\end{equation*}
	Now from \eqref{tequ1.7} and from triangle inequality for $\delta $, we have
	\begin{equation}
		\begin{array}{ll}
			&  M(x_{n+1},x_{n})  =  \max \{d(x_{n+1},x_{n}),\delta
			(x_{n+1},Sx_{n+1}),\delta (x_{n},Tx_{n}), \\
            & \hspace{4cm} {\frac{1}{2}}\left( D(x_{n+1},Tx_{n})+D(x_{n},Sx_{n+1})\right)\} \\
			& \leq  \max \{ \delta (A_{n},A_{n-1}),\delta (A_{n},A_{n+1}),\delta (A_{n-1},A_{n}), \\
			& \hspace{3cm} {\frac{1}{2}}\left( D(x_{n+1},A_{n})+\delta (A_{n-1},A_{n+1})\right) \}
			\\
			& =  \max \{ \delta (A_{n},A_{n-1}),\delta (A_{n},A_{n+1}),{\frac{1}{2}}
			\delta (A_{n-1},A_{n+1})\} \\
			& \leq  \max \{ \delta (A_{n},A_{n-1}),\delta (A_{n},A_{n+1}), {\frac{1}{2}}\left( \delta (A_{n-1},A_{n})+\delta (A_{n},A_{n+1})\right)
			\} \\
			& =  \max \left \{ \delta (A_{n-1},A_{n}),\delta (A_{n},A_{n+1})\right \}.
		\end{array}
		\label{tequ1.14}
	\end{equation}
	If $\delta (A_{n},A_{n+1})>\delta (A_{n-1},A_{n})$, then
	\begin{equation}
		M(x_{n+1},x_{n})\leq \delta (A_{n+1},A_{n}). \label{tequ1.15}
	\end{equation}
	From \eqref{tequ1.14} and \eqref{tequ1.15}), it follows that $M(x_{n+1},x_{n})=\delta
	(A_{n+1},A_{n})$ which gives
	\begin{eqnarray*}
			\psi (\delta (A_{n},A_{n+1})) & \leq & \psi (M(x_{n+1},x_{n}))-\varphi
			(M(x_{n+1},x_{n})) \\
			& < & \psi (M(x_{n+1},x_{n}))  =  \psi (\delta (A_{n},A_{n+1})),
	\end{eqnarray*}
	a contradiction. So, we have
	\begin{equation}
		\delta (A_{n},A_{n+1})\leq M(x_{n},x_{n+1})\leq \delta (A_{n-1},A_{n}). \label{tequ1.17}
	\end{equation}
	Similarly, we can also obtain inequalities \eqref{tequ1.17} in the case when $n$ is an
	even integer. Therefore, the sequence $\{a_{n}\}$ defined in \eqref{tequ1.17} is
	non-increasing and bounded below. Suppose that $a_{n}\rightarrow a$ when $
	n\rightarrow \infty $ for some $a\geq 0$. From \eqref{tequ1.17}, we have
	\begin{equation*}
		\lim\limits_{n\rightarrow \infty }\delta (A_{n},A_{n+1})=\lim\limits_{n\rightarrow \infty
		}M(x_{n},x_{n+1})=a\geq 0. 
	\end{equation*}
	From the above, one conclude that
	\begin{equation*}
		\varphi (M(x_{2n},x_{2n+1}))\leq \psi (M(x_{2n},x_{2n+1})). 
	\end{equation*}
	On taking limit on both side of above inequality, we obtain
	\begin{equation*}
		\lim\limits_{n\rightarrow \infty }\varphi (M(x_{2n},x_{2n+1}))\leq
		\lim\limits_{n\rightarrow \infty }\psi (\delta (A_{2n},A_{2n+1})).
	\end{equation*}
	Therefore, $\lim\limits_{n\rightarrow \infty }\varphi (M(x_{2n},x_{2n+1}))=0$,
	using the definition of $\varphi $, we have
	\begin{equation*}
		a=\lim\limits_{n\rightarrow \infty }M(x_{2n},x_{2n+1})=0.
	\end{equation*}
	Hence,
	$
		\lim\limits_{n\rightarrow \infty }a_{n}=\lim\limits_{n\rightarrow \infty }\delta
		(A_{n},A_{n+1})=0.
	$
	From \eqref{tequ1.11}, it follows that
	\begin{equation}
		\lim\limits_{n\rightarrow \infty }c_{n}=\lim\limits_{n\rightarrow \infty
		}d(x_{n},x_{n+1})=0. \label{tequ1.22}
	\end{equation}
	Now we show that $\{x_{n}\}$ is a Cauchy sequence. Assume to the contrary that there
	exists $\varepsilon >0$ for which we can find nonnegative integer sequences $
	\{m_{k}\}$ and $\{n_{k}\}$ such that $n_{k}$ is smallest element of the
	sequence $\{n_{k}\}$ for which $n_{k}>m_{k}>k,$
	\begin{equation*}
		\delta (A_{2m_{k}},A_{2n_{k}})\geq \varepsilon, \qquad
		\delta (A_{2m_{k}},A_{2n_{k}-2})<\varepsilon.
	\end{equation*}
	From the triangle inequality for $\delta $, we have
	\begin{eqnarray*}
			\varepsilon & \leq & \delta (A_{2m_{k}},A_{2n_{k}}) \\
			& \leq & \delta (A_{2m_{k}},A_{2n_{k}-2})+\delta
			(A_{2n_{k}-2},A_{2n_{k}-1})+\delta (A_{2n_{k}-1},A_{2n_{k}}) \\
			& < & \varepsilon +\delta (A_{2n_{k}-2},A_{2n_{k}-1})+\delta
			(A_{2n_{k}-1},A_{2n_{k}}).
	\end{eqnarray*}
	Taking limit as $k\rightarrow \infty $ and using \eqref{tequ1.22}, we conclude that
	\begin{equation*}
		\lim\limits_{k\rightarrow \infty }\delta (A_{2m_{k}},A_{2n_{k}})=\varepsilon.
	\end{equation*}
	Moreover,
	\begin{align*}
		|\delta (A_{2m_{k}},A_{2n_{k}+1})-\delta (A_{2m_{k}},A_{2n_{k}})|\leq \delta
		(A_{2n_{k}},A_{2n_{k}+1}),
	\\
		|\delta (A_{2m_{k}-1},A_{2n_{k}})-\delta (A_{2m_{k}},A_{2n_{k}})|\leq \delta
		(A_{2m_{k}},A_{2m_{k}-1}).
	\end{align*}
	Using \eqref{tequ1.22}, we get
	\begin{equation*}
		\lim\limits_{k\rightarrow \infty }\delta
		(A_{2m_{k}-1},A_{2n_{k}})=\lim\limits_{k\rightarrow \infty }\delta
		(A_{2m_{k}},A_{2n_{k}+1})=\varepsilon
	\end{equation*}
	and
	\begin{equation*}
		|\delta (A_{2m_{k}-1},A_{2n_{k}+1})-\delta (A_{2m_{k}-1},A_{2n_{k}})|\leq
		\delta (A_{2n_{k}},A_{2n_{k}+1}).
	\end{equation*}
	Using \eqref{tequ1.22}, we get
	\begin{equation*}
		\lim\limits_{k\rightarrow \infty }\delta (A_{2m_{k}-1},A_{2n_{k}+1})=\varepsilon .
	\end{equation*}
	Also, from the definition of $M$ in \eqref{tequ1.7}, we have
	\begin{equation}
		\lim\limits_{k\rightarrow \infty }M(x_{2m_{k}},x_{2n_{k}+1})=\varepsilon. \label{tequ1.32}
	\end{equation}
	Putting $x=x_{2m_{k}},y=x_{2n_{k}+1}$ in \eqref{tequ1.6}, we have
	\begin{eqnarray*}
			\psi (\delta (A_{2m_{k}},A_{2n_{k}+1})) & = & \psi (\delta
			(Sx_{2m_{k}},Tx_{2n_{k}+1})) \\
			& \leq & \psi (M(x_{2m_{k}},x_{2n_{k}+1}))-\varphi
			(M(x_{2m_{k}},x_{2n_{k}+1})).
	\end{eqnarray*}
	Thus we have
	\begin{equation}
		\varphi (M(x_{2m_{k}},x_{2n_{k}+1}))\leq \psi
		(M(x_{2m_{k}},x_{2n_{k}+1}))-\psi (\delta (A_{2m_{k}},A_{2n_{k}+1})). \label{tequ1.34}
	\end{equation}
	On which taking limit on both side of \eqref{tequ1.34} and using \eqref{tequ1.32}, we
	obtain
	\begin{equation*}
		\lim\limits_{k\rightarrow \infty }\varphi (M(x_{2m_{k}},x_{2n_{k}+1}))=0.
	\end{equation*}
	Therefore,
	$
		\varepsilon =\lim\limits_{k\rightarrow \infty }M(x_{2m_{k}},x_{2n_{k}+1})=0,
	$
	which is a contradiction. Therefore, $\{x_{n}\}$ is a Cauchy Sequence. Since $X$ is
	complete, there exists an element $u$ in $X$ such that $x_{n}\rightarrow u$
	as $n\rightarrow \infty $. Now, we show that the point $u$ is end point of $
	S $. As the limit point $u$ is independent of the choice of $x_{n}\in A_{n}$, we also get
	\begin{equation}
		\lim\limits_{n\rightarrow \infty }\delta (Sx_{2n},u)=\lim\limits_{n\rightarrow \infty
		}\delta (Tx_{2n+1},u)=0. \label{tequ1.36}
	\end{equation}
	From
	\begin{eqnarray*}
			M(u,x_{2n+1}) & = & \max \{d(u,x_{2n+1}),\delta (u,Su),\delta
			(x_{2n+1},Tx_{2n+1}), \\
			&  & {\frac{1}{2}}\left(D(u,Tx_{2n+1})+D(x_{2n+1},Su)\right) \},
	\end{eqnarray*}
	we have $M(u,x_{2n+1})\rightarrow \delta (u,Su)$ as $n\rightarrow \infty $.
	Since
	\begin{equation*}
		\psi (\delta (Su,Tx_{2n+1})\leq \psi (M(u,x_{2n+1}))-\varphi (M(u,x_{2n+1})),
	\end{equation*}
	we have
	\begin{equation*}
		\varphi (M(u,x_{2n+1}))\leq \psi (M(u,x_{2n+1}))-\psi (\delta (Su,Tx_{2n+1})).
	\end{equation*}
	Taking limit as $n\rightarrow \infty $ and using \eqref{tequ1.36}, we obtain $
	\lim\limits_{n\rightarrow \infty }\varphi (M(u,x_{2n+1}))=0$ which implies $\delta
	(u,Su)=0$ or $Su=\{u\}$. Now, we show that $u$ is also an end point for $T$.
	It is easy to see that $M(u,u)=\delta (u,Tu)$. Using that $u$ is end point
	for $S$,  have
	\begin{eqnarray}
			\psi (\delta (u,Tu))  =  \psi (\delta (Su,Tu))
			 \leq  \psi (M(u,u))-\varphi (M(u,u)) \notag \\
            =  \psi (d(u,Tu))-\varphi (d(u,Tu)), \notag
	\end{eqnarray}
	and by an argument similar to the above, we conclude that $\delta (u,Tu)=0$
	or $\{u\}=Tu$.
	
	Now assume that the end point of $S$ and $T$ are comparable. We are to show
	that $u$ is a unique common end point for $S$ and $T$. If there exists
	another point $v$ in $X$ such that $Sv=Tv=\{v\}$, then $(u,v)\in X_{\preceq
	}.$ Also $M(u,v)=d(u,v)$ and from
	\begin{eqnarray}
			\psi (d(u,v))  =  \psi (\delta (Su,Tv))  \leq  \psi (M(u,v))-\varphi (M(u,v)) \notag \\
			 =  \psi (d(u,v))-\varphi (d(u,v)), \notag
	\end{eqnarray}
	that is,
	$
			\varphi (d(u,v))\leq \psi (d(u,v))-\psi (d(u,v))
	$
	implies $\varphi (d(u,v))=0$ and we conclude that $u=v$. The proof is
	completed.
\qed \end{proof}

\begin{example}
	\label{ExampleNazSe:2.}
	Consider $X=[0,\frac{1}{4}]\times \lbrack 0,\frac{1}{4}]$, a square in the plane $\mathbb{R}^{2}$ with usual metric $d$ and partial order as stated in Example 1. Then
	$(\mathbb{R}^{2},d,\leq_{\downarrow})$ is a complete ordered $L_{\downarrow}$ metric space.
Let $S,T:X\rightarrow B(X)$ be defined as
	\begin{equation*}
		S(x_{1},x_{2})=[(0,0),(\frac{x_{1}}{4},\frac{x_{2}}{4})],\
		T(x_{1},x_{2})=[0,\frac{x_{1}}{4}]\times \lbrack 0,\frac{x_{2}}{4}]\ \text{for all}\ (x_{1},x_{2})\in X.
	\end{equation*}
	Let $\psi ,\varphi :\mathbb{R}_{\geq 0}\rightarrow \mathbb{R}_{\geq 0}$ be defined by
	\begin{equation*}
		\psi (t)=\left \{
		\begin{array}{lll}
			2t, & if & t\in \lbrack 0,\frac{1}{2}) \\
			&  &  \\
			3, & if & t\geq \frac{1}{2}.
		\end{array}
		\right.
	\end{equation*}
	Note that $\psi \in \Upsilon $ and
	\begin{equation*}
		\varphi (t)=\left \{
		\begin{array}{lll}
			\frac{t}{5}, & if & t\in \lbrack 0,5) \\
			&  &  \\
			\frac{1}{2}, & if & t\geq 5.
		\end{array}
		\right.
	\end{equation*}
	Then $\varphi (t_{n})\rightarrow 0$ implies $t_{n}\rightarrow 0$. Now, for $
	x=(x_{1},x_{2})$, $y=\left( y_{1},y_{2}\right) $ in $X$, the
	following cases arise: \newline
\noindent	(i) If $x=y$, then
	\begin{multline*}
			\psi (\delta (Sx,Ty))  =  \psi (\delta (Sx,Tx))
			 =  \psi (\frac{1}{4}\sqrt{x_{1}^{2}+x_{2}^{2}})
            = \dfrac{\sqrt{x_{1}^{2}+x_{2}^{2}}}{2}
            =\frac{1}{2}\delta (x,Sx)\\
             \leq   \frac{9}{5}M(x,y)
		      =  2M(x,y)-\frac{1}{5}M(x,y) =  \psi (M(x,y))-\varphi (M(x,y)).
	\end{multline*}
\noindent	(ii) If $x>y$, then
	\begin{multline*}
			\psi (\delta (Sx,Ty))  \leq  \psi
            \big(\sup\limits_{\substack{(x_{1},x_{2})\in Sx  \\
            (y_{1},y_{2})\in Ty}}\sqrt{(x_{1}-y_{1})^{2}+(x_{2}-y_{2})^{2}}\ \big) \\
			 \leq  \psi \big(\sup \limits_{(x_{1},x_{2})\in Sx}\sqrt{x_{1}^{2}+x_{2}^{2}}\ \big)
			=  \frac{1}{2}\delta (x,Sx)\\
            \leq  \frac{9}{5}M(x,y)
			 =2M(x,y)-\frac{1}{5}M(x,y) =  \psi (M(x,y))-\varphi (M(x,y)).
		\end{multline*}
\noindent	(iii) If $x<y$, then
	\begin{multline*}
			\psi (\delta (Sx,Ty))  \leq  \psi (\sup\limits_{\substack{(x_{1},x_{2})\in Sx  \\
                (y_{1},y_{2})\in Ty}}\sqrt{(x_{1}-y_{1})^{2}+(x_{2}-y_{2})^{2}}) \leq  \psi(\sup\limits_{(y_{1},y_{2})\in Ty}\sqrt{y_{1}^{2}+y_{2}^{2}}\ )
			\\
			 =  \frac{1}{2}\delta (y,Ty)\leq \frac{9}{5}M(x,y) = 2M(x,y)-\frac{1}{5}M(x,y) =  \psi (M(x,y))-\varphi (M(x,y)).
	\end{multline*}
Note that $M(x,y)\leq \frac{\sqrt{2}}{4}<\frac{1}{2}$ for all $x,y \in X$.
	Thus $S$ and $T$ satisfy the generalized $(\psi ,\varphi )$-weak contraction
	for all $x,y\in X$. Therefore, all the axioms of Theorem \ref{TheoremNazSe:1} are
	satisfied. Moreover, $(0,0)$ is the unique common end point in $X$.
\end{example}

 Following similar argument to that given in Theorem \ref{TheoremNazSe:1}, we can prove the following theorem.

\begin{theorem}
	\label{TheoremNazSe:2}
	Let $(X,d,\preceq_{\uparrow})$ be a complete ordered $L_{\uparrow}$ metric space. Suppose
	that $T,S:X\rightarrow B(X)$ be two partially dominating mappings that
	satisfy generalized $(\psi ,\varphi )$-weak contractive condition, where $
	\psi \in \Upsilon$, and $\varphi:\mathbb{R}_{\geq 0}\rightarrow \mathbb{R}_{\geq 0}$
	satisfies, on the sequences $\{t_n\}\subset \mathbb{R}_{\geq 0}$, the property
	\begin{equation*}
		(\varphi (t_{n})\rightarrow 0\  \ as \ n\rightarrow \infty) \ \ \ \Rightarrow\  \  \ (t_{n}\rightarrow 0 \ as \ n\rightarrow \infty).
	\end{equation*}
	Then, there exists the point $u\in X$ such that $\{u\}=Tu=Su$. Moreover, if
	the end points of $S$ and $T$ are comparable, then $S$ and $T$ have a unique
	common end point.
\end{theorem}

\begin{corollary}
	\label{CorollaryNazSe:1}
	\ Let $(X,d,\preceq_{\downarrow})$ be a complete ordered $L_{\downarrow}$ and $
	T,S:X\rightarrow B(X)$ be two partially dominated mappings which satisfy
	\begin{gather*}
		\psi (\delta (S^{k}x,T^{k}y))\leq \psi (M(x,y))-\varphi (M(x,y)),  \\
 	M(x,y)=\max \left \{ d(x,y),\delta (x,S^{k}x),\delta (y,T^{k}y),{\frac{1}{2}}
		( D(x,T^{k}y)+D(y,S^{k}x)) \right \}
	\end{gather*}
	for all $x,y\in X_{\preceq_{\downarrow}}$ and
$k\in\mathbb{N},$ with $\psi \in \Upsilon $ and
$\varphi :\mathbb{R}_{\geq 0}\rightarrow \mathbb{R}_{\geq 0}$ satisfying, on the sequences $\{t_n\}\subset \mathbb{R}_{\geq 0}$, the property 	
    \begin{equation*}
		(\varphi (t_{n})\rightarrow 0\  \ as \ n\rightarrow \infty) \ \ \ \Rightarrow\  \  \ (t_{n}\rightarrow 0 \ as \ n\rightarrow \infty).
	\end{equation*}
	Then, there exists the point $u\in X$ such that $\{u\}=Tu=Su$. Moreover, if
	the end points of $S^{k}$ and $T^{k}$ are comparable, then $S$ and $T$ have
	a unique common end point.
\end{corollary}

\begin{proof} It follows from Theorem  \ref{TheoremNazSe:2}, that
$S^{k}u=T^{k}u=\{u\}$ and $u$ is a unique common end point for $S^{k}$ and $T^{k}.$ Now
	\begin{equation*}
		\{Su\} = SS^{k}u=S^{k+1}u=S^{k}Su,\quad \{Tu\} = TT^{k}u=T^{k+1}u=T^{k}Tu
	\end{equation*}
	imply that $Su\ $and $Tu$ are also end points of $S^{k}\ $and $T^{k}$. By
	the uniqueness, $u$ is unique common end point of $S$ and $T$.
\qed \end{proof}

\begin{corollary}
	\label{CorollaryNazSe:2}
	Let $(X,d,\preceq_{\downarrow})$ be a complete ordered $L_{\downarrow}$ metric space and let $
	T,S:X\rightarrow B(X)$ be two partially dominated mappings which satisfy
	\begin{equation*}
		\delta (Sx,Ty)\leq \lambda \max \{d(x,y),\delta (x,Sx),\delta (y,Ty),
{\frac{1}{2}}(D(x,Ty)+D(y,Sx)) \},
	\end{equation*}
	for all $x,y\in X_{\preceq_{\downarrow}}$ and $0 \leq \lambda <1.$ Then, there
	exists the point $u\in X$ such that $\{u\}=Tu=Su$. Moreover, if the end
	points of $S$ and $T$ are comparable, then $S$ and $T$ have a common end
	point.
\end{corollary}

\begin{proof} \smartqed Define $\varphi ,\psi :\mathbb{R}_{\geq 0}\rightarrow \mathbb{R}_{\geq 0}$ by $\psi (t)=t$ and $\varphi (t)=(1-\lambda )t$ for all $t\in	\mathbb{R}_{\geq 0}.$ The result follows from Theorem \ref{TheoremNazSe:2}.
\qed \end{proof}

\begin{corollary}
	\label{CorollaryNazSe:3}	Let $(X,d,\preceq_{\downarrow})$ be a complete ordered $L_{\downarrow}$ metric space and $
	T,S:X\rightarrow B(X)$ be two partially dominated mappings which satisfy
	\begin{equation*}
		\psi (\delta (Sx,Ty))\leq \psi (d(x,y))-\varphi (d(x,y))
	\end{equation*}
	for all $x,y\in X_{\preceq_{\downarrow}}$ and $k\in \mathbb{N},$ where $\psi ,\varphi :\mathbb{R}_{\geq 0}\rightarrow \mathbb{R}_{\geq 0},$ $\psi $
	is continuous monotone nondecreasing, $\varphi $ is a lower semi-continuous,
	and $\psi (t)=\varphi (t)=0$ if and only if $t=0.$ Then, there exists the
	point $u\in X$ such that $\{u\}=Tu=Su$. Moreover, if the end points of $S$
	and $T$ are comparable, then $S$ and $T$ have a unique common end point.
\end{corollary}

\begin{proof} \smartqed
	The result follows from Theorem \ref{TheoremNazSe:2}.
\qed \end{proof}

\begin{corollary}
	\label{CorollaryNazSe:4}
	Let $(X,d,\preceq_{\downarrow})$ be a complete ordered $L_{\downarrow}$ metric space and $
	T,S:X\rightarrow B(X)$ be two partially dominated mappings which satisfy
	\begin{equation*}
		\delta (Sx,Ty)\leq \frac{d(x,y)}{d(x,y)+1}
	\end{equation*}
	for all $x,y\in X_{\preceq_{\downarrow}}.$ Then, there exists the point $u\in X$
	such that $\{u\}=Tu=Su$. Moreover, if the end points of $S$ and $T$ are
	comparable, then $S$ and $T$ have a unique common end point.
\end{corollary}

\begin{proof} \smartqed Define $\varphi ,\psi :\mathbb{R}_{\geq 0}\rightarrow \mathbb{R}_{\geq 0}$ by $\psi (t)=t$ and $\varphi (t)=\frac{t^{2}}{t+1}$ for all $t\in \mathbb{R}_{\geq 0}.$ The result follows from Corollary \ref{CorollaryNazSe:3}.
\qed \end{proof}

 If we take $T=S$, then we have the following corollary.

\begin{corollary}
	\label{CorollaryNazSe:5}
	Let $(X,d,\preceq_{\downarrow})$ be a complete ordered $L_{\downarrow}$ metric space and $
	T:X\rightarrow B(X)$ be a partially dominated map which satisfies
	\begin{equation*}
		\delta (Tx,Ty)\leq \frac{d(x,y)}{d(x,y)+1}
	\end{equation*}
	for all $x,y\in X_{\preceq_{\downarrow}}.$ Then, there exists the point $u\in X$
	such that $\{u\}=Tu$. Moreover, if the end points of $T$ are comparable,
	then $T$ has a unique end point.
\end{corollary}

 A single-valued selfmap $f$ on $X$ is called dominated map if $fx\preceq x$
for each $x\in X$.

\begin{corollary}
	\label{CorollaryNazSe:6}
	Let $(X,d,\preceq_{\downarrow})$ be a complete ordered $L_{\downarrow}$ metric space and $
	f,g:X\rightarrow X$ be two partially dominated mappings satisfy
	\begin{equation}
			\max \left \{
			\begin{array}{lll}
				d(fx,gx), & d(fx,gy), &  \\ 				
                d(fy,gx), & d(fy,gy) &
			\end{array}
			\right \}  \leq  \frac{d(x,y)}{d(x,y)+1} 		
		\label{tCorolNazSe:6}
	\end{equation}
	for all $x,y\in X_{\preceq_{\downarrow}}$. Then, there exists the point $u\in X$ such
	that $u=fu=gu$. Moreover, if the common fixed points of $f$ and $g$ are
	comparable, then $f$ and $g$ have a unique common fixed point.
\end{corollary}

\begin{proof} \smartqed We can define $Tx=\{fx,gx\}$. From \eqref{tCorolNazSe:6}, we conclude that
	\begin{equation*}
		\delta (Tx,Ty)\leq \frac{d(x,y)}{d(x,y)+1}.
	\end{equation*}
	Now, we apply Corollary \ref{CorollaryNazSe:5} to obtain that $T$ has a unique end point $
	x_{0}$. Therefore, $Tx_{0}=\{x_{0}\}$. Hence $fx_{0}=gx_{0}=x_{0}$.
\qed \end{proof}



\end{document}